\documentclass[preprint,12pt]{elsarticle}

\usepackage[utf8]{inputenc}
\usepackage{amsthm,amsmath,amssymb,amsfonts}
\usepackage{epsfig}
\usepackage{mathrsfs}
\usepackage{graphicx,color}
\usepackage{eucal}
\usepackage{enumerate}
\usepackage{randtext}

\usepackage{bbm}
\theoremstyle{definition}
\newtheorem{defn}{Definition}[section]
  
\theoremstyle{plain}
\newtheorem{thm}{Theorem}

\theoremstyle{remark}
\newtheorem{rmk}{Remark}[section]

\newcommand{\R}{\mathbb{R}}

\newcommand{\Cur}{C}

\newcommand\restr[1]{\raisebox{-.5ex}{$|$}_{#1}}

\newcommand\restrcB{_{(t,\epsilon)}}

\newcommand\indi[1]{\mathbb{I}_{\left\{#1\right\}}}

\numberwithin{equation}{section}

\font\eka=cmex10
\def\ind{\mathrel{\hbox{\rlap{%
\hbox to 7.5pt{\hrulefill}}\raise6.6pt\hbox{\eka\char'167}}}}

\journal{SI of Journal of Multivariate Analysis}

\begin{document}

\begin{frontmatter}

%% Title, authors and addresses

\title{Integrated shape-sensitive functional metrics}

%% use the tnoteref command within \title for footnotes;
%% use the tnotetext command for theassociated footnote;
%% use the fnref command within \author or \address for footnotes;
%% use the fntext command for theassociated footnote;
%% use the corref command within \author for corresponding author footnotes;
%% use the cortext command for theassociated footnote;
%% use the ead command for the email address,
%% and the form \ead[url] for the home page:
%% \title{Title\tnoteref{label1}}
%% \tnotetext[label1]{}
%% \author{Name\corref{cor1}\fnref{label2}}
%% \ead{email address}
%% \ead[url]{home page}
%% \fntext[label2]{}
%% \cortext[cor1]{}
%% \address{Address\fnref{label3}}
%% \fntext[label3]{}

\author{Sami Helander\fnref{lb1}}  %\ead{sami.helander@aalto.fi}
\author{Petra Laketa\fnref{lb2}}  %\ead{laketa@karlin.mff.cuni.cz}
\author{Pauliina Ilmonen\fnref{lb1}}  %\ead{pauliina.ilmonen@aalto.fi}
\author{Stanislav Nagy$^\ast$\fnref{lb2}}  %\ead{nagy@karlin.mff.cuni.cz}
\author{Germain Van Bever\fnref{lb3}}  %\ead{germain.vanbever@unamur.be}
\author{Lauri Viitasaari\fnref{lb4}} %\ead{lauri.viitasaari@aalto.fi}

\cortext[cor1]{\textit{Correspondence:} Stanislav Nagy, \textit{Email: \randomize{nagy@karlin.mff.cuni.cz}}}

%% use optional labels to link authors explicitly to addresses:
%% \author[label1,label2]{}
%% \address[label1]{}
%% \address[label2]{}

\address[lb1]{Aalto University School of Science, Finland}
\address[lb2]{Charles University, Czech Republic}
\address[lb3]{University of Namur, Belgium}
\address[lb4]{Aalto University School of Business, Finland}

\begin{abstract}
This paper develops a new integrated ball (pseudo)metric which provides an intermediary between a chosen starting (pseudo)metric $d$ and the $L_p$ distance in general function spaces. Selecting $d$ as the Hausdorff or Fr\'echet distances, we introduce integrated shape-sensitive versions of these supremum-based metrics. The new metrics allow for finer analyses in functional settings, not attainable applying the non-integrated versions directly. Moreover, convergent discrete approximations make computations feasible in practice.
\end{abstract}

% %%Graphical abstract
% \begin{graphicalabstract}
% %\includegraphics{grabs}
% \end{graphicalabstract}

% %%Research highlights
% \begin{highlights}
% \item Research highlight 1
% \item Research highlight 2
% \end{highlights}

\begin{keyword}
Fr\'echet distance \sep Functional data analysis \sep Hausdorff distance  \sep Pseudometric.
%% keywords here, in the form: keyword \sep keyword
%% PACS codes here, in the form: \PACS code \sep code
%% MSC codes here, in the form: \MSC code \sep code
%% or \MSC[2008] code \sep code (2000 is the default)
\end{keyword}

\end{frontmatter}

%%%%%%%%%%%%%%%%%%%%%%%%%%%
%%%%%%%%%%%%%%%%%%%%%%%%%%%
%%%%%%%%%%%%%%%%%%%%%%%%%%%

\section{Introduction}

Measuring the shape similarity or dissimilarity of functions has been a concurrent problem in many fields of research. In functional data analysis, the shape features often comprise the key modes of variance in the functional observations and comparison problems such as classification are common. This can be seen, for instance, in the functional depth literature, where recent advances have begun emphasizing shape features as the focal point of analysis. 
See, for example \citet{SgueraEtal2014}, \citet{ClaeskensEtal2014}, \citet{elastidepth} and \citet{NagyEtal2017, NagyEtal2020} for recent approaches to shape-sensitive functional depths. 

Similarly, attention has been devoted to developing metrics that are more sensitive to variations in shape. The use of such metrics has become commonplace in shape and pattern matching applications in machine learning and computer vision. 
See \citet{HuttenlocherEtal1993}, \citet{Rucklidge1997}, \citet{YiAndCamps1999}, \citet{VeltkampAndHagedoorn2001}, \citet{DeCarvalhoEtal2006}, \citet{AltAndGodau1992}, \citet{AltEtal2003}, \citet{BrakatsoulasEtal2005}, \citet{AronovEtal2006} and \citet{JiangEtal2008} for a wide variety of applications for such shape-sensitive metrics.
Popular approaches include considering Hausdorff or Fr\'echet distances between the graphs of the functions. However, as both of these distances are based on a global supremum, they do not consider the local likeness of the graphs. Therefore, in the literature some attention has been devoted to developing averaged (i.e. integrated) versions of these distances. 

\citet{Baddeley1992} developed an integrated alternative to the Hausdorff metric by substituting the supremum for an integral in an expression equivalent to the usual definition. 
A more direct approach was taken by \citet{MarronAndTsybakov1995} who proposed 
an error criterion for nonparametric curve estimation, where the features of the curves are matched ``vertically'' in value and ``horizontally'' in timing. 
\citet{SchutzeEtal2012} proposed an averaged version of the Hausdorff distance by replacing the \textit{sup-inf} in its definition with an $L_p$ mean. 
\citet{VargasAndBogoya2018} and \citet{BogoyaEtal2018} further developed the notion to an integrated counterpart. 

The integrated Fr\'echet distance was studied in \citet{Buchin2007} who considered several approaches for both a summed and an averaged Fr\'echet distance and analyzed their properties. 
\citet{EfratEtal2007} developed a notion of dynamic time warping that can be seen as an integrated version of the Fr\'echet distance. \citet{BuchinEtal2009} considered the partial curve matching problem where the similarity of two curves was measured by the largest fraction of their lengths that could be matched within a certain Fr\'echet distance threshold. 

Many of the approaches discussed above, however, do not satisfy the key properties of a metric. See for example \citet{MarronAndTsybakov1995}, \citet{MaheshwariEtal2018}, and the extensive analysis provided by \citet{Buchin2007} for more details. Specifically, in the case of integrated Fr\'echet distances, the triangle inequality fails for approaches considered previously.

Use of pseudometrics (also called semimetrics) in FDA has been shown to be of interest. See, for example, \citet{AneirosEtAl2019b,AneirosEtAl2019a}, \citet{Cuevas2014} and \citet{GoiaAndVieu2016}. In this paper, we generalize any pseudometric $d$ on a function space to a family of integrated ball pseudometrics $d^{\epsilon,p}$, obtained through local integration of $d$. The locality parameter $\epsilon$ allows to control the shape sensitivity of the pseudometric by balancing between $d$ (for $\epsilon\to \infty$) and the $L_p$ distance (for $\epsilon\to 0$). Furthermore, we show that $d^{\epsilon,p}$ is a pseudometric for any $\epsilon$. Convergence of suitable discretizations ensures that the integrated ball pseudometric can be computed in practice. When the integrated distance $d$ is a metric, then, under mild technical assumptions, the integrated version is a metric as well.

The theoretical results are derived for general (pseudo)metrics $d$. In a second part of this paper, we apply the integrated ball construction to the Hausdorff and Fr\'echet distances, $d_H$ and $d_F$, and provide their locally integrated counterparts, $d_H^{\epsilon,p}$ and $d_F^{\epsilon,p}$. We show important properties of the integrated ball Hausdorff and Fr\'echet distances and study their behavior in a variety of simulated outlier detection settings. 
The appeal of local distances in the functional context is illustrated in Figure \ref{fig:Motivating}, where examining local distances $d^{\epsilon,p}_H$ and $d^{\epsilon,p}_F$, $p=2$, $\epsilon\in[0,1]$, highlights shape-outlying curves. In this example, the compactness of the domain $[0,1]$ of the functional data allows to restrict $\epsilon$ to $[0,1]$. 

\begin{figure}[!ht]
    \centerline{\includegraphics[scale=0.45]{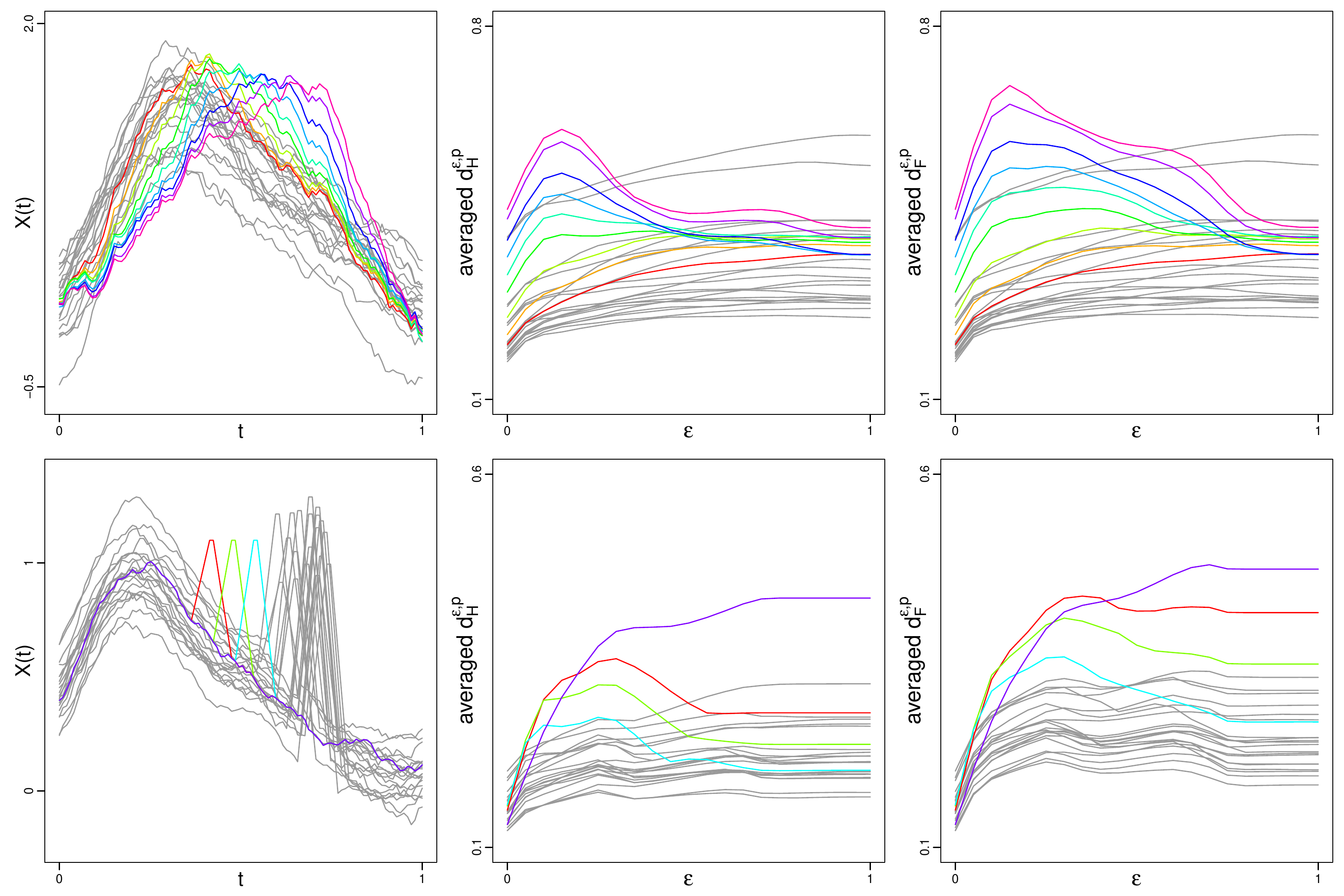}}
    \caption{Motivating example showcasing the behavior of the integrated ball metrics $d_H^{\epsilon,p}$ and $d_F^{\epsilon,p}$ as a function of $\epsilon\in[0,1]$, $p=2$, in two simulated settings with shape-outlying curves (in color).}
    \label{fig:Motivating}
\end{figure}

Figure \ref{fig:Motivating} illustrates two models, drawn in gray in the left panels, with some contaminating observations drawn in color. 
On the top row, the outlying observations differ from the base model in phase, whereas on the bottom, three of the outliers have different timing of the sharp feature and one lacks it entirely. In the middle and right panels, the average $d_H^{\epsilon,p}$ and $d_F^{\epsilon,p}$ distances between the observations and each outlier are drawn in corresponding color, as a function of $\epsilon$, with the between-observation average distances drawn in light gray. 

In the first model, global Hausdorff and Fr\'echet distances (corresponding to $\epsilon=1$) fail to properly separate the outliers from one another and from the rest of the observations. In the second model, the $L_p$ distance (for $\epsilon=0$) does not properly separate outlying curves and fares poorly in detecting them among the observations of the base model. Examining the curves $\epsilon\mapsto d_H^{\epsilon,p}$ and $\epsilon\mapsto d_F^{\epsilon,p}$ shows a different behavior for the outlying curves and allows proper separation on a wide range of values of the locality parameter $\epsilon$.

The rest of the paper is organized as follows: in Section \ref{sec:integrated-ball-metric} we introduce a general functional framework, under which we present the main contribution of the paper, the integrated ball (pseudo)metric $d^{\epsilon,p}$, and analyze its properties. 
We show that, under non-restrictive assumptions, $d^{\epsilon,p}$ is indeed a proper (pseudo)metric with many desirable properties. Furthermore, we provide an important convergence result for $d^{\epsilon,p}$ for discrete approximations of the functions, that ensures the feasibility of computations of the (pseudo)metric in practice. In Section~\ref{sec:examples}, we formally introduce our two leading examples, the Hausdorff and Fr\'echet distances $d_H$ and $d_F$, and define their locally integrated counterparts, $d_H^{\epsilon,p}$ and $d_F^{\epsilon,p}$. We show that $d_H$ and $d_F$ satisfy our main assumptions and therefore $d_H^{\epsilon,p}$ and $d_F^{\epsilon,p}$ exhibit the properties stated in Section~\ref{sec:integrated-ball-metric}. In Section \ref{sec:simulated} we apply the integrated ball Hausdorff and Fr\'echet distances in simulated outlier detection settings, and study their behavior as a function of the locality parameter $\epsilon$. A discussion closes the paper.

%%%%%%%%%%%%%%%%%%%%%%%%%%%
%%%%%%%%%%%%%%%%%%%%%%%%%%%
%%%%%%%%%%%%%%%%%%%%%%%%%%%

\section{Integrated ball metric} \label{sec:integrated-ball-metric}
The framework considered in this paper is the following. Let $(\mathcal{X},d_\mathcal{X})$ and $(\mathcal{V},d_{\mathcal{V}})$ be two metric spaces. Throughout, we assume that $\mathcal{X}$ is a complete and a measure space with respect to a Borel probability measure $\mu$. Let $\mathcal{F}$ be a family of measurable and bounded functions $f:\mathcal{X} \to \mathcal{V}$. Recall that $f$ is bounded if the image $f(\mathcal{X})=\{f(x):x\in\mathcal{X}\}$ is a bounded subset of $\mathcal{V}$. As we consider bounded functions, $\mathcal{F}$ can be equipped with the usual supremum metric given by
$
\sup_{t\in \mathcal{X}}d_{\mathcal{V}}\left(f(t),g(t)\right).
$
The \emph{curve} associated with $f$ is denoted by $C_f$ and is defined as
$$C_f:\mathcal{X}\rightarrow\mathcal{X}\times\mathcal{V}:t\mapsto (t,f(t)).$$
The \emph{graph} of $f$ is the image $G_f=C_f(\mathcal{X})\subset \mathcal{X}\times\mathcal{V}$. 

For a function $f$, the restriction of $f$ to the set $A\subset\mathcal{X}$ is the function $f\restr{A}:{A}\to\mathcal{V}$. In the sequel, by restriction of curve/graph of $f$ to $A$, we mean the curve/graph associated with the restriction of $f$ to $A$. 
Let $B(t,r)=\{x\in\mathcal{X}:d_\mathcal{X}(x,t)\leq r\}\subset\mathcal{X}$ denote the closed ball of radius $r$ centered at $t$. The restriction $f\restr{B(t,\epsilon)}$ will be written $f_{(t,\epsilon)}$ throughout.

In order to take into account shape-sensitive features that are not detected by classical ``vertical'' measures depending only on pointwise distances $d_{\mathcal{V}}(f(t),g(t))$ in the image space $\mathcal{V}$, we consider a general distance $d$ defined on $\mathcal{X} \times \mathcal{V}$. 
More precisely, we assume that $d$ is a pseudometric\footnote{Recall that a mapping $d:\mathcal{B} \times \mathcal{B} \to [0,\infty]$ on general sets $\mathcal{B}$ is a pseudometric if it is symmetric ($d(b_1,b_2)=d(b_2,b_1)$ for all $b_1,b_2\in\mathcal{B}$), satisfies triangle inequality, and $d(b,b) = 0$ for all $b \in \mathcal{B}$. Note that some authors (see, e.g.~\citet{FerratyAndVieu2006}) refer to such a mapping as a semimetric. We will use the term \emph{pseudometric} throughout.} 
defined for subsets of $\mathcal{X} \times \mathcal{V}$ of the form
$$
\mathcal{B}:=\{B(t,r) \times f(B(t,r)) :\ t\in\mathcal{X}, \ r>0,\ f\in\mathcal{F}\}.
$$
Elements of $\mathcal{B}$ are graphs of functions restricted to closed balls. 
\begin{rmk}
We stress that, while the underlying spaces $(\mathcal{X},d_\mathcal{X})$ and $(\mathcal{V},d_{\mathcal{V}})$ are metric, the distance $d$ that we use throughout is only required to be a pseudometric on the product space $\mathcal{X}\times\mathcal{V}$. Note also that care must be taken on the definition of $d$ when one considers which subsets of $\mathcal{X}\times \mathcal{V}$ are accounted for. Indeed, it might be that $d$ provides a pseudometric when restricted to a suitable class of subsets of $\mathcal{X} \times \mathcal{V}$ while it is not if one considers all possible subsets. An example is the Hausdorff distance which is a metric only if restricted to non-empty compact subsets of $\mathcal{X} \times \mathcal{V}$. 
\end{rmk}

Note that assuming that $d$ is defined on sets in $\mathcal{X}\times \mathcal{V}$ is natural. While there exist many definitions computing distances between functions $f\in\mathcal{F}$ ($L_p$ distances) or curves $C_f$ (Fr\'echet distance), they all can be interpreted as (pseudo)metrics on subsets of $\mathcal{X}\times \mathcal{V}$ due to the bijective nature of the mappings $f\mapsto C_f$ and $f\mapsto G_f$. In the sequel, we will use the notation $d(f,g)$, $d(C_f,C_g)$ and $d(G_f,G_g)$ interchangeably.

The running example that is used throughout considers the set of continuous functions $\mathcal{F}=\mathcal{C}([0,1],\R^m)=\{f:[0,1]\rightarrow\mathbb{R}^m,\ f\textrm{ continuous}\}$. 
Euclidean distances are used on both $\mathcal{X}=[0,1]$ and on $\mathcal{V}=\mathbb{R}^m$ and the Lebesgue measure makes the compact interval $[0,1]$ a measure space. 
The canonical metric on $\mathcal{F}$ is induced by the supremum distance. 
Then the pseudometric $d$ can be induced by the Hausdorff or Fr\'echet distances, see Section \ref{sec:examples}.

\subsection{Definition}

We start with introducing a family of local versions of $d$, the integrated ball metrics, which involve integrating distances between restricted functions. 
\begin{defn}[Integrated ball metric]\label{def:integral-ball-metric}
    Let $d$ be a pseudometric on subsets of $\mathcal{X}\times\mathcal{V}$ and let $\mu$ be a Borel probability measure on $\mathcal{X}$. For a given $\epsilon>0$ and $1\leq p\leq \infty$, define the \emph{integrated ball metric} $d^{\epsilon,p}$ between $f,g\in\mathcal{F}$ as
    \begin{equation*}
	d^{\epsilon,p}(f,g) := \begin{cases}
	\left(\int_\mathcal{X} d(f_{(t,\epsilon)},g_{(t,\epsilon)})^pd\mu(t) \right)^{\frac{1}{p}}&\textrm{ for } p<\infty,\\[2mm]
	\sup_{t\in\mathcal{X}}	d(f_{(t,\epsilon)},g_{(t,\epsilon)})& \textrm{ for }p=\infty.
	\end{cases}
	\end{equation*}
\end{defn}

Note that, provided the integrand is measurable, the integral in Definition~\ref{def:integral-ball-metric} exists and takes value in $[0,\infty]$.
The integrated ball metric provides a flexible approach to construct local distances between functions. 
The next section shows that the locality parameter $\epsilon$ allows to consider distances ranging from the $L_p$ distance (for $\epsilon=0$) to $d$ itself (for $\epsilon\rightarrow\infty$). 
Similarly, freedom on the choice in the order $p$ and in the measure $\mu$ provides flexibility towards different applications. 
One can emphasize specific regions of $\mathcal{X}$ through the measure $\mu$, or adjust penalization for large deviations in small regions of space through the choice of $p$.

\subsection{General properties}

We now examine the continuity properties of $d^{\epsilon,p}$ and the assumptions required to ensure it being a (pseudo)metric. Our assumptions are collected below.

\paragraph{Assumptions}
\begin{enumerate}
    \item[(A1)] For $\epsilon>0$ fixed, for each $t\in \mathcal{X}$ it holds $\mu(B(t,\epsilon))>0$.
    \item[(A2)] There exists $K>0$ such that, for all $t\in\mathcal{X}$ and $\epsilon\geq 0$,
    $$d(f_{(t,\epsilon)},g_{(t,\epsilon)}) \leq K \sup\limits_{x\in\mathcal{X}}\ d_\mathcal{V}( f(x),g(x)).$$
    \item[(A3)] For any $t\in\mathcal{X}$, it holds $\lim_{\epsilon\to 0}d(f_{(t,\epsilon)},g_{(t,\epsilon)}) = d_\mathcal{V}( f(t),g(t) ).$
    \item[(A4)] For any $t\in\mathcal{X}$ and for any sequence $\epsilon_n\rightarrow\epsilon$, it holds $d(f_{(t,\epsilon_n)},f_{(t,\epsilon)})\rightarrow 0.$
  	\item[(A5)] For any $\epsilon>0$ and for any sequence $t_n\rightarrow t$, it holds $d(f_{(t_n,\epsilon)},f_{(t,\epsilon)})\rightarrow 0.$
\end{enumerate}

Assumptions (A1)--(A5) above are mild and merely technical. 
Assumption (A1) is required for identifiability of pointwise values of the functions in this general setting and is not even required if one considers functions as equivalence classes $\mu$-almost everywhere. 
(A2) is a regularity condition on $d$ that supposes that two functions $f$ and $g$ cannot be arbitrarily far if their images are close (as measured via $d_{\mathcal{V}}$). 
In the example of continuous functions, this translates to an assumption that diagonal distances between points in the graphs of functions are bounded by a multiple of the vertical distances between functions. 
At the same time, (A2) ensures that the dominated convergence theorem can be applied.  
In conjunction with (A2), assumption (A3) ensures that $\lim_{\epsilon\to 0}d^{\epsilon,p}$ is indeed the $L^p$ distance on $\mathcal{F}$. 
(A4) and (A5) are natural continuity assumptions on $d$. (A4) ensures continuity of $d^{\epsilon,p}$ in $\epsilon$, while (A5) ensures continuity of the integrand in $t$. 

We now state the main result for this section.

\begin{thm}\label{thm:integral-ball-key-properties}
    For any pseudometric $d$ on $\mathcal{B}$, for $1\leq p\leq \infty$ and for $\epsilon> 0$, the integrated ball construction $d^{\epsilon,p}$ produces a pseudometric. Furthermore, $d^{\epsilon,p}$ becomes a true metric if $\mu$ satisfies (A1) and $d$ is, for any $t \in \mathcal{X}$, a metric on $\{B(t,\epsilon)\times f(B(t,\epsilon)):f\in\mathcal{F}\}$.
    Moreover, for any $f,g \in \mathcal{F}$, the following holds.
    \begin{itemize}
         \item[(i)] Under (A2) is the map $d^{\epsilon,p}(\cdot,f)\equiv d^{\epsilon,p}(f,\cdot): \mathcal{F} \to [0,\infty]$ continuous.
        %\item $d_B^{\epsilon,p}$ satisfies $d_B^{\epsilon,p}\leq\sup_{s\in[0,1]}K\Vert f(s)-g(s)\Vert_\mathcal{V}$, for some $K>0$.
        \item[(ii)] Under (A2) and assuming that $\mathcal{X}$ is bounded, $\lim_{\epsilon\rightarrow \infty}d^{\epsilon,p}(f,g)= d(f,g).$
        \item[(iii)] Under (A2) and (A3), 
        $$\lim_{\epsilon\rightarrow 0}d^{\epsilon,p}(f,g)=  
        \begin{cases}
\left(\int_\mathcal{X}  d_\mathcal{V}( f(t),g(t))^p d\mu(t)\right)^{\frac{1}{p}} & \textrm{for }1\leq p<\infty,\\[2mm]
\sup_{t\in\mathcal{X}}d_{\mathcal{V}}(f(t),g(t))&\textrm{for }p=\infty,        	
        \end{cases}
$$ 
that is, $d^{\epsilon,p}(f,g)$ converges to the $L_p$ distance on $\mathcal{F}$ as $\epsilon \to 0$.
        \item[(iv)] Under (A2) and (A4) the mapping $d^{\,\cdot\,,p}:[0,\infty)\rightarrow\mathbb{R}:\epsilon\mapsto d^{\epsilon,p}(f,g)$ is continuous.
    \end{itemize}
\end{thm}
We show in Section \ref{sec:examples} that the Hausdorff and Fr\'echet distances satisfy~(A1)--(A5) and consequently, Theorems~\ref{thm:integral-ball-key-properties} and \ref{thm:discretisation-convergence} (below) are applicable. 
The integrated ball distance construction appears to be the first instance of a local version of the Fr\'echet distance that provides a continuous metric. This is also the case for the Hausdorff distance in the context of functional data.  

\begin{rmk}
	The boundedness assumption on $\mathcal{X}$ is only a sufficient condition for $(ii)$ of Theorem~\ref{thm:integral-ball-key-properties}. Its proof shows that the statement remains valid more generally provided that $d(f_{(t,\epsilon)},g_{(t,\epsilon)}) \to d(f,g)$ as $\epsilon \to \infty$, under further  technical assumptions. 
\end{rmk}

\begin{proof}[Proof of Theorem \ref{thm:integral-ball-key-properties}]
We prove the statements one by one and divide the proof into four steps.\\

\noindent\textbf{Step 1a: $d^{\epsilon,p}$ is a pseudometric.}\\
Since $d$ is a pseudometric on $\mathcal{B}$ and $f_{(t,\epsilon)},g_{(t,\epsilon)}\in \mathcal{B}$ for every $t\in\mathcal{X}$ and $\epsilon>0$, it follows trivially that $d^{\epsilon,p}(f,g)$ is symmetric, non-negative, and satisfies $d^{\epsilon,p}(f,f) =0$ for all $f \in \mathcal{F}$. 
In order to prove the triangle inequality for $d^{\epsilon,p}$, we observe from the triangle inequality for $d(f_{(t,\epsilon)},g_{(t,\epsilon)})$ that, for all $t\in \mathcal{X}, \epsilon>0$, and $h \in \mathcal{F}$, we have
	\begin{equation*}
	d(f\restrcB,g\restrcB) \leq d(f\restrcB,h\restrcB)+d(h\restrcB,g\restrcB).
	\end{equation*}
The claim for $p=1$ now follows by integration, and the general case $p>1$ from Minkowski's inequality. This verifies the triangle inequality, and hence $d^{\epsilon,p}$ is a pseudometric. \\

\noindent\textbf{Step 1b: $d^{\epsilon,p}$ is a metric.}\\
Under the further assumption (A1) and the hypothesis that $d$ is a metric, suppose that $d^{\epsilon,p}(f,g) = 0$, i.e. 
$$\int_\mathcal{X} d(f\restrcB,g\restrcB)^pd\mu(t) = 0.$$ It follows that
	$$
d(f\restrcB,g\restrcB) = 0 \hspace{0.5cm} \text{for }\mu-\text{almost every }t\in\mathcal{X},
	$$
	and since $d$ is a metric on $\left\{B(t,\epsilon)\times f(B(t,\epsilon))\colon f\in\mathcal{F}\right\}$ for every $t\in\mathcal{X}$, we have 
	\begin{equation*}
	%\label{eq:thm1}
		f(s) =g(s),\quad \forall s \in {B(t,\epsilon)} 
	\end{equation*}
	for $\mu$-a.e. $t\in \mathcal{X}$. Let $A$ be the set of points $t$ satisfying the last formula. It holds $\mu(A^c)=0$. 
	This implies that $f(s)=g(s)$ for all $s\in\mathcal{X}$. Indeed, if that was not the case, then there exists $s$ such that $B(s,\epsilon)\subset A^c$, which contradicts (A1). This concludes step 1. \\
	
%	Since the result holds for any $t\in\mathcal{X}$ and $\epsilon>0$, A covering argument together with (A1)  implies $f(s)=g(s)$ for all $s\in \mathcal{X}$, concluding step 1.\\
	
\noindent\textbf{Step 2: Proof of statement (i).}\\
	For simplicity of notation, we prove only the case $p=1$, while the general case can be proved similarly.
	Using (A2), we get
$$
	d^{\epsilon,p}(f,g) = \int_\mathcal{X} d(f\restrcB,g\restrcB)d\mu(t)\leq K\sup_{s\in\mathcal{X}} d_\mathcal{V}( f(s), g(s) ).
$$
Now, let $f_n$ be a sequence of functions such that $\sup_{t \in \mathcal{X}} d_{\mathcal{V}}(f_n(t),f(t)) \to 0$ as $n\rightarrow \infty$. Then, by the reverse triangle inequality, it holds that
$$
 |d^{\epsilon,p}(f_n,g)-d^{\epsilon,p}(f,g)| \leq d^{\epsilon,p}(f,f_n) \leq K\sup_{s\in \mathcal{X}}  d_\mathcal{V}( f(s), f_n(s) )\rightarrow 0.
$$
	The continuity in $g$ can be proved similarly, concluding step 2.\\

\noindent\textbf{Step 3: Proof of statements (ii) and (iii).}\\
By observing that (A2) allows us to use the dominated convergence theorem, statement (iii) follows directly from (A3). Similarly for statement (ii), boundedness of $\mathcal{X}$ gives $d(f\restrcB,g\restrcB) = d(f\restr{\mathcal{X}},g\restr{\mathcal{X}})=d(f,g)$ for $\epsilon$ large enough, from which the claim follows trivially.\\

\noindent\textbf{Step 4: Proof of statement (iv).}\\
	Let $t$ be fixed. By the reverse triangle inequality we get
		$$
	\left|d(f\restrcB, g_{(t,\epsilon+\delta)})-d(f_{(t,\epsilon+\delta)},g_{(t,\epsilon+\delta)})\right| \leq d(f_{(t,\epsilon+\delta)}, f\restrcB).
	$$
	Treating $g$ similarly by symmetry, we get 
$$\left|d(f_{(t,\epsilon+\delta)}, g_{(t,\epsilon+\delta)})-d(f\restrcB, g\restrcB)\right|\leq d(f_{(t,\epsilon+\delta)}, f\restrcB)+d(g_{(t,\epsilon+\delta)}, g\restrcB).
$$
	The claim follows directly from (A4) by using the dominated convergence theorem. This concludes step 4 and the whole proof.
\end{proof}

%%%%%%%%%%%%%%%%%%%%%%%%%%%
%%%%%%%%%%%%%%%%%%%%%%%%%%%
%%%%%%%%%%%%%%%%%%%%%%%%%%%

\subsection{Computation}

In practice, functions are only observed on a discrete grid, so that it is impossible to compute the exact value of the integrated ball distance between two data points $f$ and $g$. In this section, we construct approximations $d^{\epsilon,p}_n(f,g)$ of $d^{\epsilon,p}(f,g)$ that are consistent in $n$.

More precisely, let $\mu_n$ be a sequence of probability measures converging weakly to $\mu$. Let $d_n$ be a sequence of pseudometrics, which, for any $f$ and $g$ provides a sequence $d_n(f_{(t,\epsilon)},g_{(t,\epsilon)})$ of approximations of $d(f_{(t,\epsilon)},g_{(t,\epsilon)})$. In the practical setting, $\mu_n$ corresponds to a discrete measure on the observational grid. The pseudometrics $d_n$ are discrete approximations of $d$ based on the observed data points. The practical construction of $d_n$ depends on the underlying $d$ and specific examples are deferred to Section \ref{sec:simulated}.

The following result ensures the consistency of suitable discretizations in practice. 
\begin{thm}
\label{thm:discretisation-convergence}
For $\epsilon> 0$ and $1\leq p\leq \infty$, suppose that (A2) and (A5) hold, that $\mu_n$ converges weakly to $\mu$ and that, as $n\to \infty$,
\begin{equation}
    \label{eq:discretisation-uniform-conv}
\sup_{t\in \mathcal{X}} \left|d_{n}(f_{(t,\epsilon)},g_{(t,\epsilon)}) - d(f_{(t,\epsilon)},g_{(t,\epsilon)})\right| \to 0.
\end{equation}
Then, defining $d_n^{\epsilon,p}(f,g)$ as $d^{\epsilon,p}(f,g)$ with $\mu=\mu_n$ and $d=d_n$, it holds that
$$
d_{n}^{\epsilon,p}(f,g) \to d^{\epsilon,p}(f,g).
$$
\end{thm}

\begin{proof}[Proof of Theorem \ref{thm:discretisation-convergence}]
We consider $p<\infty$; the case $p=\infty$ is trivial. We have
\begin{align*}
\left[d_{n}^{\epsilon,p}(f,g)\right]^p
=& \int_{\mathcal{X}} \Big[ d_{n}\left(f_{(t,\epsilon)},g_{(t,\epsilon)}\right)^p-d\left(f_{(t,\epsilon)},g_{(t,\epsilon)}\right)^p \Big]d\mu_n(t)\\
&\qquad +\int_{\mathcal{X}} d\left(f_{(t,\epsilon)},g_{(t,\epsilon)}\right)^p d\mu_n(t).
\end{align*}
%$$
%\left[d_{n}^{\epsilon,p}(f,g)\right]^p
%= \int_{\mathcal{X}} \Big[ d_{n}\left(f_{(t,\epsilon)},g_{(t,\epsilon)}\right)^p-d\left(f_{(t,\epsilon)},g_{(t,\epsilon)}\right)^p \Big]d\mu_n(t) +\int_{\mathcal{X}} d\left(f_{(t,\epsilon)},g_{(t,\epsilon)}\right)^p d\mu_n(t).
%$$
Using the arguments of step 4 in the proof of Theorem \ref{thm:integral-ball-key-properties} together with (A5) shows that, for fixed $f,g\in \mathcal{F}$, 
the mapping $t\mapsto d(f_{(t,\epsilon)},g_{(t,\epsilon)})$ is continuous and, by (A2), bounded. Hence, by definition of weak convergence \cite{Billingsley1999}, 
$$
\int_{\mathcal{X}} d\left(f_{(t,\epsilon)},g_{(t,\epsilon)}\right)^p d\mu_n(t)  \to \int_{\mathcal{X}} d\left(f_{(t,\epsilon)},g_{(t,\epsilon)}\right)^p d\mu(t).
$$
Furthermore, by our assumption \eqref{eq:discretisation-uniform-conv} it holds that
%\begin{align*}
%&\left|\int_{\mathcal{X}} \left[ d_{n}\left(f_{(t,\epsilon)},g_{(t,\epsilon)}\right)^p-d\left(f_{(t,\epsilon)},g_{(t,\epsilon)}\right)^p \right] d\mu_n(t)\right|\\[2mm]
%&\qquad\qquad\qquad\leq \ \sup_{t\in \mathcal{X}}\left|d_{n}\left(f_{(t,\epsilon)},g_{(t,\epsilon)}\right)^p-d\left(f_{(t,\epsilon)},g_{(t,\epsilon)}\right)^p\right|\to 0,
%\end{align*}
$$
\begin{aligned}
\left|\int_{\mathcal{X}} d_{n}\left(f_{(t,\epsilon)},g_{(t,\epsilon)}\right)^p \right. & - d\left(f_{(t,\epsilon)},g_{(t,\epsilon)}\right)^p  d\mu_n(t)\bigg\vert \\
& \leq \ \sup_{t\in \mathcal{X}}\left|d_{n}\left(f_{(t,\epsilon)},g_{(t,\epsilon)}\right)^p-d\left(f_{(t,\epsilon)},g_{(t,\epsilon)}\right)^p\right|\to 0,
\end{aligned}
$$
and hence $\left[d_{n}^{\epsilon,p}(f,g)\right]^p \to \left[d^{\epsilon,p}(f,g)\right]^p$ which concludes the proof.
\end{proof}

\begin{rmk}
The proof of Theorem~\ref{thm:discretisation-convergence} can be generalized to incorporate measurement errors in $f$ and $g$. More precisely, the proof holds under the assumption 
$$\sup_{t\in \mathcal{X}} \left|d_{n}(f_{n,(t,\epsilon)},g_{n,(t,\epsilon)}) - d(f_{(t,\epsilon)},g_{(t,\epsilon)})\right| \to 0,$$
where $f_n$ and $g_n$ are approximations of $f$ and $g$ with measurement errors. 
\end{rmk}

%%%%%%%%%%%%%%%%%%%%%%%%%%%
%%%%%%%%%%%%%%%%%%%%%%%%%%%
%%%%%%%%%%%%%%%%%%%%%%%%%%%

\section{Examples} \label{sec:examples}

In the following, we study the more specific examples of integrated Hausdorff and Fr\'echet distances. 
Let $(\mathcal{X},\mu)=([0,1],\lambda)$ with $\lambda$ the Lebesgue measure and let $(\mathcal{V}, d_\mathcal{V}) = (\R^m,d_\mathcal{V})$ with $d_\mathcal{V}$ the Euclidean distance. 
The definitions of Hausdorff and Fr\'echet distances require the use of a metric between points on curves, i.e. on $[0,1]\times\mathbb{R}^m$. We endow this space with the usual $q$-metric 
\begin{equation} \label{eq:q-metric}
    d_\Cur((t,v), (s,w)) := 
    \begin{cases}
\left(|t-s|^q + d_\mathcal{V}( v,w )^q \right)^\frac{1}{q}& \textrm{for }1\leq q<\infty,\\[2mm]     	
\max\left\{|t-s|, d_\mathcal{V}( v,w ) \right\} & \textrm{for }q=\infty.
    \end{cases}
\end{equation}
Choosing $q=2$ leads to $d_\Cur$ becoming the metric induced by the $L_2$ norm on $\R^{m+1}$. 
For the sake of simplicity, in the sequel we restrict to $\mathcal{C}([0,1],\mathbb{R}^m)$, the space of continuous functions $[0,1]\to\mathbb{R}^m$. 
The continuity ensures compactness of $f([0,1])$. This in turn allows us to restrict our considerations to non-empty compact subsets of $[0,1]\times \mathbb{R}^m$, for which the Hausdorff distance is a metric. \\

%\subsection{Integrated ball Hausdorff distance} \label{subsec:integral-ball-Hausdorff}
Definition~\ref{def:integral-ball-metric} allows to turn any distance defined on restricted functions to a local integrated version. We start by providing such a distance between restrictions in the Hausdorff case.

\begin{defn}[Hausdorff for restrictions]
    \label{def:Hausdorff}
    Let $f\restr{A}$ and $g\restr{B}$ be the restrictions of the functions $f,g\in\mathcal{C}([0,1],\mathbb{R}^m)$ to sets $A$ and $B$, respectively. Then, the \emph{Hausdorff distance $d_H$ for the restrictions} is given by
    \begin{align*}
    d_H(f\restr{A},g\restr{B}) :=\max\{ \sup_{t\in A}\inf_{s\in B} d_\Cur(\Cur_f(t), \Cur_g(s)),\ \sup_{s\in B}\inf_{t\in A}d_\Cur(\Cur_f(t), \Cur_g(s))\}.
    \end{align*}
\end{defn}

The integrated version of the Hausdorff distance follows.
\begin{defn}[Integrated ball Hausdorff distance]
	% \label{defn:integral-ball-Hausdorff}
	For a given $\epsilon\geq 0$, $1\leq p\leq \infty$ and $f,g\in\mathcal{C}([0,1],\mathbb{R}^m)$, we define
	\begin{equation*}
	d_H^{\epsilon,p}(f,g) := 
	\begin{cases}
\left(\int_0^1 d_H(f\restrcB,g\restrcB)^pdt \right)^{\frac{1}{p}} & \textrm{for }1\leq p<\infty,\\[2mm]
	\sup_{t\in[0,1]} d_H(f\restrcB,g\restrcB)	&\textrm{for }p=\infty.
	\end{cases}
	\end{equation*}
\end{defn}

%%%%%%%%%%%%%%%%%%%%%%%%%%%
%%%%%%%%%%%%%%%%%%%%%%%%%%%
%%%%%%%%%%%%%%%%%%%%%%%%%%%

%\subsection{Integrated ball Fr\'echet distance}\label{subsec:integral-ball-Frechet}

The Fr\'echet distance is defined using parametrizations of the curves of the functions. In the case of the Fr\'echet distance, care must therefore be taken in how the restrictions of the curves are handled. Recall the usual one-sided definition of the Fr\'echet distance from~\citet{Buchin2007}.% considered in the literature.

\begin{defn}[One-sided Fr\'echet]\label{def:oneSidedFrechet}
	Let $\Phi$ be the set of all continuous monotonically increasing functions $\phi:[0,1]\mapsto[0,1]$ with $\phi(0)=0$ and $\phi(1)=1$. Then the \emph{one-sided Fr\'echet distance} $d_F$ between the functions $f,g\in\mathcal{C}([0,1],\mathbb{R}^m)$ is given by
	$$
	d_F(f,g):=\inf_{\phi\in\Phi} \sup_{t\in[0,1]} d_\Cur(\Cur_f(t),\Cur_g(\phi(t))).
	$$
\end{defn}

In order to define a restricted version of this distance satisfying (A4) and (A5), one needs to compare functions defined on sets that do not necessarily overlap. 
To do so, the two domains are mapped to a shared interval, to ensure that both curve segments can be simultaneously traversed through a single parameter. 
Therefore, for $0\leq a\leq b\leq 1$, we denote
\begin{equation*} % \label{eq:F-tilde}
    \tilde f_{[a,b]}:[0,1]\to \mathcal{V}:t\mapsto f(\theta_{[a,b]}(t)),
\end{equation*}
where $\theta_{[a,b]}:[0,1]\to [a,b]:t\mapsto a+(b-a)t$. This construction will be typically applied to function $f$ restricted to the interval $[a,b] \subset [0,1]$; the function $\tilde f_{[a,b]}$ is then a re-parametrization of $f\restr{[a,b]}$ to the full domain $[0,1]$. Evidently, $\tilde f_{[0,1]} = f$. We also denote $\tilde f_{[a,b]}$ for $[a,b] = B(t,\epsilon)$ by the simpler $\tilde f_{(t,\epsilon)}$. 
This leads us to a slightly expanded definition of the Fr\'echet distance that is able to compare curve segments restricted to different intervals.

\begin{defn}[One-sided Fr\'echet for restrictions]
    \label{def:Frechet-differentRestrictions}
     Let $f\restr{[a,b]}$ and $g\restr{[c,d]}$ be the restrictions of the functions $f, g\in\mathcal{C}([0,1],\mathbb{R}^m)$ to the respective subintervals $[a,b]$ and $[c,d]$ of $[0,1]$. Then, the \emph{one-sided Fr\'echet distance $d_F$ for restrictions} is given by
    \begin{equation*}
        d_F(f\restr{[a,b]},g\restr{[c,d]}) := d_F(\tilde f_{[a,b]}, \tilde g_{[c,d]}).
    \end{equation*}
\end{defn}
In the case $[a,b]=[c,d]=[0,1]$, Definitions~\ref{def:Frechet-differentRestrictions} and~\ref{def:oneSidedFrechet} coincide. %For common intervals $[a,b]=[c,d]$ we get
%\begin{equation} \label{eq:Frechet-restrictionsSimple}
   % d_F(f\restr{[a,b]},g\restr{[a,b]}) := \inf\limits_{\phi\in\Phi}\sup\limits_{s\in [a,b]}d_\Cur(\Cur_f(s),\Cur_g(\phi(s))),
%\end{equation}
%where $\Phi$ is the space of all continuous monotonically increasing functions $\phi:[a,b]\to[a,b]$, with $\phi(a)=a$ and $\phi(b)=b$. 
%Note that \eqref{eq:Frechet-restrictionsSimple} also coincides with Definition \ref{def:oneSidedFrechet}, as one can consider functions on any interval $[a,b]$ equivalently to functions on $[0,1]$, through the simple re-parametrization $f(\theta_{[a,b]}(t))$.
With this definition of the Fr\'echet distance, the integrated ball Fr\'echet distance can now be defined.
\begin{defn}[Integrated ball Fr\'echet distance]
    % \label{def:integral-ball-Frechet}
	For $\epsilon\geq 0$, $1\leq p\leq \infty$ and $f,g\in\mathcal{C}([0,1],\mathbb{R}^m)$, define
	\begin{equation*}
	d_F^{\epsilon,p}(f,g) := \begin{cases}
 	\left(\int_0^1 d_F(f_{(t,\epsilon)},g_{(t,\epsilon)})^pdt \right)^{\frac{1}{p}}&\textrm{for }1\leq p<\infty,\\[2mm]
 	\sup_{t\in[0,1]} d_F(f\restrcB,g\restrcB)&\textrm{for }p=\infty.
 \end{cases}
	\end{equation*}
\end{defn}
Both $d_H^{\epsilon,p}$ and $d_F^{\epsilon,p}$ inherit the properties of Theorems \ref{thm:integral-ball-key-properties} and \ref{thm:discretisation-convergence}.
\begin{thm} \label{prop:integral-Hausdorff-Frechet-properties}
    $d_H$ and $d_F$ are metrics that fulfill (A1)--(A5). Thus, for any $\epsilon\geq 0$ and $1\leq p\leq\infty$, $d_H^{\epsilon,p}$ and $d_F^{\epsilon,p}$ are metrics that satisfy all conclusions of Theorem \ref{thm:integral-ball-key-properties} and Theorem \ref{thm:discretisation-convergence}.
\end{thm}

\begin{proof}[Proof of Theorem \ref{prop:integral-Hausdorff-Frechet-properties}]
    We check that the Hausdorff and Fr\'echet distances, $d_H$ and $d_F$, satisfy assumptions (A1)--(A5) in our chosen setting. Then, the statements of Theorem \ref{thm:integral-ball-key-properties} and Theorem \ref{thm:discretisation-convergence} follow directly. \\
    
\noindent\textbf{The Hausdorff case:}\\
Note that the graphs of continuous functions $f\in\mathcal{C}([0,1],\R^m)$ are compact sets in $[0,1]\times\R^m$. Therefore, the Hausdorff distance defines a metric for the restrictions of the graphs $f\restr{A}$. In turn, $d_H^{\epsilon,p}$ is a metric.\\
    
	\noindent (A1): Follows directly from our choice of $\mu$ being the Lebesgue measure.\\
	
	\noindent (A2): We prove a stronger statement: for all $t\in[0,1]$ and $\epsilon>0$,
	$$
	d_H(f\restrcB, g\restrcB) \leq \sup_{s\in{B(t,\epsilon)}} d_\mathcal{V}(f(s),g(s)).
	$$
    From the definition of $d_H$ and using $d_\Cur$ the $q$-metric \eqref{eq:q-metric}, for any $a\in{B(t,\epsilon)}$, it holds that
    $$
        \inf_{b\in{B(t,\epsilon)}} d_\Cur( \Cur_f(a),\Cur_g(b) ) \leq d_\Cur( \Cur_f(a),\Cur_g(a) )=  d_\mathcal{V}( f(a),g(a) ),
    $$
    from which it follows that
    $$
    \sup_{a\in{B(t,\epsilon)}}\inf_{b\in{B(t,\epsilon)}} d_\Cur( \Cur_f(a),\Cur_g(b) ) \leq \sup_{a\in{B(t,\epsilon)}}  d_\mathcal{V}( f(a),g(a) ).
    $$
    The claim now follows by considering the definition of the Hausdorff distance and interchanging the roles of $f$ and $g$. \\

	\noindent(A3): The first inequality in the proof of (A2) above, for continuous $f$ and $g$, implies that
	$$
	\limsup_{\epsilon\rightarrow0} d_H(f\restrcB,g\restrcB) \leq  d_\mathcal{V}( f(t),g(t) ).
	$$
	For the lower bound, using Definition \ref{def:Hausdorff} and $d_\Cur$ the $q$-metric, it holds that
%	\begin{align*}
%    	d_H(f\restrcB,g\restrcB) &\geq \inf_{s\in{B(t,\epsilon)}} d_\Cur( \Cur_f(t),\Cur_g(s) )\\
%    	&=  \inf_{s\in{B(t,\epsilon)}} \left[ |t-s|^q + d_\mathcal{V}( f(t),g(s) )^q \right]^\frac{1}{q}.
%	\end{align*}
$$
d_H(f\restrcB,g\restrcB) \geq \inf_{s\in{B(t,\epsilon)}} d_\Cur( \Cur_f(t),\Cur_g(s) )=  \inf_{s\in{B(t,\epsilon)}} \left[ |t-s|^q + d_\mathcal{V}( f(t),g(s) )^q \right]^\frac{1}{q}.
$$
	By continuity of $f$ and $g$, letting $\epsilon\rightarrow0$, it follows that
	$$
	\inf_{s\in{B(t,\epsilon)}}\left[ |t-s|^q + d_\mathcal{V}( f(t),g(s) )^q \right]^\frac{1}{q} \rightarrow  d_\mathcal{V}( f(t),g(t) ).
	$$
	The proof for $q = \infty$ is analogous. \\
     
	 \noindent (A4) and (A5): Consider two restrictions $f_{(t_1,\epsilon_1)}$ and $f_{(t_2,\epsilon_2)}$ of a single function $f$. 
    Obviously,
    \begin{align*}
        d_H(f_{(t_1,\epsilon_1)}, f_{(t_2,\epsilon_2)}) =\max &\left\{\sup_{a\in B{(t_1,\epsilon_1)}\setminus B{(t_2,\epsilon_2)}}\inf_{b\in B{(t_2,\epsilon_2)}}d_C(C_f(a),C_f(b) ), \right. \\
        &\left. \qquad\sup_{a\in B{(t_2,\epsilon_2)}\setminus B{(t_1,\epsilon_1)}}\inf_{b\in B{(t_1,\epsilon_1)}}d_C(C_f(a),C_f(b) )\right\}.
    \end{align*}
 %$$
%d_H(f_{(t_1,\epsilon_1)}, f_{(t_2,\epsilon_2)}) =\max \left\{\sup_{a\in B{(t_1,\epsilon_1)}\setminus B{(t_2,\epsilon_2)}}\inf_{b\in B{(t_2,\epsilon_2)}}d_C(C_f(a),C_f(b) ), \right. \left. \qquad\sup_{a\in B{(t_2,\epsilon_2)}\setminus B{(t_1,\epsilon_1)}}\inf_{b\in B{(t_1,\epsilon_1)}}d_C(C_f(a),C_f(b) )\right\}.
%$$
	%
    From the continuity of $f$ we know that for each $\varepsilon_f>0$ there is $\delta>0$ such that $d_C(C_f(a),C_f(b))\leq \varepsilon_f$ when $\vert a-b\vert \leq\delta$.\\
    
    Fix $t_1=t_2=t$, with $\epsilon_1=\epsilon_n$ and $\epsilon_2=\epsilon$. 
    Suppose, w.l.o.g. that $\epsilon_n>\epsilon$. 
    Then $B{(t,\epsilon)}\subset B{(t,\epsilon_n)}$. 
    Consider $a\in B{(t,\epsilon_n)}\setminus B{(t,\epsilon)}$. 
    There is a point $b\in B{(t,\epsilon)}$ such that $\vert a-b \vert \leq \vert \epsilon_n-\epsilon \vert$. 
    For $n$ large enough, we have $\vert \epsilon_n-\epsilon \vert<\delta$ and consequently $d_H(f_{(t,\epsilon_n)}, f_{(t,\epsilon)})\leq \varepsilon_f$. This shows (A4).\\
    
    Fix now $\epsilon_1=\epsilon_2=\epsilon$, with $t_1=t_n$ and $t_2=t$. If $a\in B{(t_n,\epsilon)}\setminus B{(t,\epsilon)}$ or $a\in B{(t,\epsilon)}\setminus B{(t_n,\epsilon)}$, there is $b$ from $B{(t,\epsilon)}$ or $B{(t_n,\epsilon)}$, respectively, satisfying $\vert a-b\vert \leq \vert t_n-t\vert$. Again, if $n$ is chosen large enough so that $\vert t_n-t\vert<\delta$, the desirable inequality $d_H(f_{(t_n,\epsilon)}, f_{(t,\epsilon)})\leq \varepsilon_f$ follows. This shows (A5).\\

\noindent\textbf{The Fr\'echet case:}\\
Note that $d_F$ is a metric, and thus so is $d_F(f\restrcB,g\restrcB)$ for every $t$ and any $\epsilon$. In turn, $d_F^{\epsilon,p}$ is a metric.\\
     
    \noindent (A1): The claim follows directly as for the Hausdorff distance.\\
    
    \noindent (A2): Again, we prove a stronger statement: for all $t\in[0,1]$ and $\epsilon>0$,
	$$
	d_F(f\restrcB, g\restrcB) \leq \sup_{s\in{B(t,\epsilon)}} d_\mathcal{V}(f(s),g(s)).
	$$ 
	%
	%Consider the equivalent expression \eqref{eq:Frechet-restrictionsSimple} to the definition of $d_F$. 
	Denote by ${\tilde{\Cur}_f}$ and $\tilde{\Cur}_g$ the curves corresponding to $\tilde{f}\restrcB$ and $\tilde{g}\restrcB$, respectively (with the ball $B(t,\epsilon)$ properly restricted to $[0,1]$). The choice of $d_\Cur$ to be the $q$-metric \eqref{eq:q-metric} yields
    \begin{align*}
        d_F(f\restrcB,g\restrcB) &= \inf\limits_{\phi\in\Phi} \sup\limits_{s\in [0,1]} d_\Cur(\tilde{\Cur}_f(s),\tilde{\Cur}_g(\phi(s))) 
         \leq \sup\limits_{s\in [0,1]} d_\Cur( \tilde{\Cur}_f(s), \tilde{\Cur}_g(s))\\
        & = \sup\limits_{s\in [0,1]} d_\mathcal{V}( \tilde{f}_{(t,\epsilon)}(s), \tilde{g}_{(t,\epsilon)}(s) ) = \sup\limits_{s\in B(t,\epsilon)} d_\mathcal{V}( f(s),g(s)),
    \end{align*}
    and the claim follows directly.\\
    
    \noindent (A3): The reasoning in the proof of (A2) above, for continuous $f$ and $g$, implies that
    $$
    \limsup_{\epsilon\rightarrow0}d_F(f\restrcB,g\restrcB)\leq  d_\mathcal{V}( f(t),g(t)).
    $$
    For the lower bound, fix $t \in [0,1]$. By the continuity of both $f$ and $g$, for any $\eta > 0$ there exists $\epsilon > 0$ such that $\left\vert s - t \right\vert \leq \epsilon$ implies $d_\mathcal{V}( f(t), f(s)) \leq \eta$ and $d_\mathcal{V}( g(t), g(s)) \leq \eta$. For any $s_1, s_2 \in [0,1]$ with $\left\vert s_1 - t \right\vert \leq \epsilon$ and $\left\vert s_2 - t \right\vert \leq \epsilon$ we obtain 
			\[	
			\begin{aligned}
			d_\mathcal{V}( f(t), g(t)) & \leq d_\mathcal{V}( f(t), f(s_1) ) + d_\mathcal{V}( f(s_1), g(s_2) ) + d_\mathcal{V}( g(s_2), g(t)) \\
			& \leq 2 \, \eta + d_\mathcal{V}(f(s_1),g(s_2)).	
			\end{aligned}
			\]
%$$
%d_\mathcal{V}( f(t), g(t)) \leq d_\mathcal{V}( f(t), f(s_1) ) + d_\mathcal{V}( f(s_1), g(s_2) ) + d_\mathcal{V}( g(s_2), g(t))  \leq 2 \, \eta + d_\mathcal{V}(f(s_1),g(s_2)).	
%$$
		For any $s_1, s_2 \in B(t,\epsilon)$, we get that $d_\mathcal{V}(f(s_1),g(s_2)) \geq d_\mathcal{V}( f(t), g(t)) - 2\,\eta$, meaning that when restricted to the ball $B(t,\epsilon)$, the graphs of the functions $f$ and $g$ are at least $(d_\mathcal{V}( f(t), g(t)) - 2\,\eta)$-separated. Definition~\ref{def:Frechet-differentRestrictions} directly gives that $d_F(f_{(t,\epsilon)},g_{(t,\epsilon)}) \geq d_\mathcal{V}( f(t), g(t)) - 2\,\eta$. Taking $\eta \to 0$, we obtain
		\[	\liminf_{\epsilon\rightarrow0} d_F(f\restrcB,g\restrcB) \geq d_\mathcal{V}( f(t), g(t)),	\]
		which together with the bound for the upper limit yields the result. \\
		%as $\epsilon\rightarrow0$, $|\phi(s)-s|\rightarrow0$ for all $s\in{B(t,\epsilon)}$. Thus, by continuity of $f$ and $g$, 
    %\begin{align*}
        %d_F(f\restrcB,g\restrcB)&= \inf\limits_{\phi\in\Phi}\sup\limits_{s\in {B(t,\epsilon)}}d_\Cur(\Cur_f(s),\Cur_g(\phi(s))) \geq \inf\limits_{\phi\in\Phi} d_\Cur(\Cur_f(t),\Cur_g(\phi(t))),
    %\end{align*}
 %and consequently,
    %$$
    %\inf\limits_{\phi\in\Phi} d_\Cur(\Cur_f(t),\Cur_g(\phi(t))) \rightarrow d_\Cur(\Cur_f(t),\Cur_g(t))= d_\mathcal{V}( f(t),g(t)).
    %$$
    
		%{\color{red} 
		\noindent (A4) and (A5): For any intervals $[a,b],[c,d]\subset[0,1]$ we can write for $s\in[0,1]$
			\[ \left\vert \theta_{[c,d]}(s)-\theta_{[a,b]}(s) \right\vert = \left\vert (c-a)(1-s)+(d-b)s \right\vert \leq \max\{\vert c-a\vert,\vert d-b\vert\}.	\]
		Consequently, for any $s\in [0,1]$, $\epsilon \geq 0$, and $\delta \geq 0$
		\begin{equation*}
		\vert \theta_{{B(t,\epsilon)}}(s)-\theta_{{B(t,\epsilon+\delta)}}(s)\vert \leq \delta \quad \mbox{ and } \quad \vert \theta_{{B(t,\epsilon)}}(s)-\theta_{{B(t+\delta,\epsilon)}}(s)\vert \leq \delta.
		\end{equation*}
		By Definition~\ref{def:Frechet-differentRestrictions} and the definition of the $q$-metric \eqref{eq:q-metric} it holds that
    \begin{align*}
    d_F (f\restrcB,f_{(t,\epsilon+\delta)}) & = \inf_{\phi\in\Phi} \sup_{s\in[0,1]} d_\Cur( \tilde f\restrcB(s) , \tilde f_{(t,\epsilon+\delta)}(\phi(s)) )\\
    & \leq \sup_{s\in[0,1]} d_\Cur( \tilde f\restrcB(s) , \tilde f_{(t,\epsilon+\delta)}(s) ) \\
 %   \end{align*}
%	 From $d_\Cur$ being the $q$-metric in \eqref{eq:q-metric}, we have for $q < \infty$
%    \begin{align*}
%        &\sup_{s\in[0,1]} d_\Cur( \tilde f\restrcB(s) , \tilde f_{(t,\epsilon+\delta)}(s) )\\
        & = \sup_{s\in[0,1]} d_\mathcal{V}( f(\theta_{{B(t,\epsilon)}}(s)),f(\theta_{{B(t,\epsilon+\delta)}}(s)) ) \leq \varepsilon_f. % [ |\theta_{{B(t,\epsilon)}}(s)-\theta_{{B(t,\epsilon+\delta)}}(s)|^q + d_\mathcal{V}( f(\theta_{{B(t,\epsilon)}}(s)),f(\theta_{{B(t,\epsilon+\delta)}}(s)) )^q ]^\frac{1}{q}\\
    %    & \leq\ [\delta^q +\varepsilon_f^q]^\frac{1}{q},  
    \end{align*}
The last inequality uses our first derived bound and the continuity of $f$, which gives that $d_\mathcal{V}( f(s_1),f(s_2) ) \leq \varepsilon_f$ whenever $|s_1-s_2|\leq\delta$. The claim now follows by letting $\delta\rightarrow0$, since in that case we may write $\varepsilon_f \to 0$. This concludes (A4). To prove (A5), one uses the same argument and the second inequality above. 
\end{proof}

\begin{rmk}
    As an interesting special case, consider $[0,1]\times\mathcal{V}$ equipped with the $L_{\infty}$ distance
    $$
    d_C(\Cur_f(t), \Cur_g(s)) = \max\{|t-s|, \Vert f(t)-g(s) \Vert_\mathcal{V} \},
    $$
    where $\Vert\cdot\Vert_\mathcal{V}$ is a norm on $\mathcal{V}$. Then, the one-sided Fr\'echet distance is equivalent to the Skorokhod distance (see~\citet{Billingsley1999}) between the curves $\Cur_f$ and $\Cur_g$ (see \citet{MajumdarAndPrabhu2015}), and consequently $d_F^{\epsilon,p}$ yields an integrated ball version of the Skorokhod distance.
\end{rmk}

%%%%%%%%%%%%%%%%%%%%%%%%%%%
%%%%%%%%%%%%%%%%%%%%%%%%%%%
%%%%%%%%%%%%%%%%%%%%%%%%%%%

\section{Simulated examples} \label{sec:simulated}

Hausdorff and Fr\'echet distances  are commonly used in pattern matching and outlier detection applications.
In this section, we study the behavior of the associated integrated ball metrics in a range of simulated examples considered in the literature. 
The distances $d_H^{\epsilon,p}$ and $d_F^{\epsilon,p}$ are computed in four different settings:  three models described in \citet{DaiEtal2020}, together with an additional model designed in the spirit of the motivating example of \citet{MarronAndTsybakov1995}. Throughout this section, we use $p=2$.

Each of the models presents different types of outliers controlled by outlyingness parameters. For each outlier type, we study the behavior of $d_H^{\epsilon,p}$ and $d_F^{\epsilon,p}$ as a function of the locality parameter $\epsilon$, over a range of parameter combinations. 
The four outlier models are described below. 
In each model, $e_X(t)$ is a centered Gaussian process with covariance function $K(s,t) = \mathrm{cov}\{e_X(s), e_X(t)\} = \exp\{-|s-t|\}$.\\

\noindent\textbf{Model 1 (Jump):} The base model is $X(t) = 4t + e_X(t),\ t\in[0,1]$. The outlier model is $Y(t) = 4t + H_1\,\indi{t>0.7} + e_X(t)$, where $\mathbb{I}$ is the indicator function, and the parameter $H_1$, controlling the height of the jump, was chosen on a uniform grid in $[0,2]$.

\smallskip

\noindent\textbf{Model 2 (Peak):} The base model is the same as in Model 1. The outlier model is $Y(t) = 4t + H_2\,\indi{0.5 - W\leq t \leq 0.5+W} + e_X(t)$. The parameters $H_2\in[0.5,3]$ and $W\in[0,0.2]$, controlling the height and width of the peak (respectively), were chosen on uniform grids on their respective intervals.

\smallskip

\noindent\textbf{Model 3 (Phase):} The general model is $X(t) = S_{P}(t) + 0.9\, e_X(t),\ t\in[0,1]$, where $S_{P}(t)$ is the function with a peak of height $6$ at $t=P$, generated by smoothing (using a Gaussian smoother) the piecewise linear function interpolating the points $(0,0)$, $(P,6)$ and $(1,0)$.
For the base model, the peak is located at $P=0.3$. The outlier model is $Y(t) = S_{P}(t) + 0.9\, e_X(t),\ t\in[0,1]$, where the peak location parameter $P$ was chosen from a uniform grid on $[0.3,0.7]$.

\smallskip

\noindent\textbf{Model 4 (Feature misalignment):} The general model is 
$$X(t) = \indi{|t-T|>0.06}(S(t)+0.25\,e_X(t)) + \indi{|t-T|\leq 0.06}L_{T,H}(t),$$ where $S(t)$ is the function with a peak of height $3$ generated by smoothing the piecewise linear function interpolating the points $(0,2)$, $(0.2,3)$ and $(1,0)$ and $L_{T,H}(t)$ is a piecewise linear function that generates a peak of height $H$ at $T$. The function $L_{T,H}(t)$ interpolates the points $(T-0.06,S(T-0.06)+0.25\,e_X(T-0.06))$, $(T,H)$ and $(T+0.06,S(T+0.06)+0.25\,e_X(T+0.06))$. For the base model, the peak times $T$ followed the uniform distribution on $[0.7, 0.85]$ and peak heights $H$ followed the uniform distribution on $[3,4]$. For the outlier model, the peak height was $3.5$ and the peak times $T$ were chosen from a uniform grid on $[0.3,0.8]$.

\smallskip

%\noindent\textbf{Model 5 (Slope):} The base model is $X(t) = A+B_X\arctan(t)+e_X(t)$, where the variables $A$ and $B_X$ are independently uniformly distributed on $[-2,2]$ and $[1,2]$ respectively, and $e_X(t)$ is a centered Gaussian process with covariance function $K(s,t) = cov\{e_X(s), e_X(t)\} = 0.1\exp\{-|s-t|/0.3\}$. The outlier model is $Y(t) = 0.5 + B_Y\arctan(t)+e_X(t)$, where the slope parameter $B_Y$ was chosen on a uniform grid on $[0,2]$.
%
%
%\noindent\textbf{Model 6 (Oscillation):} The base model is $X(t) = U_{1}cos(2\pi t) + U_{2}sin(2\pi t)$, where $U_{1}$, $U_{2}$ are independent uniform variables on $[0,1]$. The outlier model is $Y(t) = \tilde U_{1}cos(2\pi t) + \tilde U_{2}sin(2\pi t), \ t\in[0,1]$, where $\tilde U_{1}$ and $\tilde U_{2}$ were chosen from uniform grids on $[0,1.6]$.

%
%\begin{figure}[!ht]
%    \centerline{\includegraphics[scale=0.45]{Figures/All-Models.pdf}}
%    \caption{Realizations of the six simulated models with an example set of contaminating outliers. The observations from the base model are drawn in black, while the colored curves represent the contaminating outliers drawn over a grid of some combination of the outlyingness parameters.}
%    \label{fig:All-models}
%\end{figure}

%\newpage
For each model, we simulated $20$ realizations from a base model and added the outliers for different parameter combinations. 
For each of the outliers, we computed their average distances $d_H^{\epsilon,p}$ and $d_F^{\epsilon,p}$ to the observations of the base models, for $\epsilon\in[0,1]$. 
For comparison, we also computed the corresponding average distances $d_H^{\epsilon,p}$ and $d_F^{\epsilon,p}$ between the observations, without the presence of the contaminating outliers. 
Figure~\ref{fig:Models} displays the realizations, together with the average integrated ball Hausdorff and Fr\'echet distances as a function of $\epsilon$.
For a cleaner comparison of the outlier types, the same realizations were used in Models $1$ and $2$ for the observations from the base model as well as the error function $e_X$ of the outlier. 

\begin{figure}[!ht]
    \centerline{\includegraphics[scale=0.4]{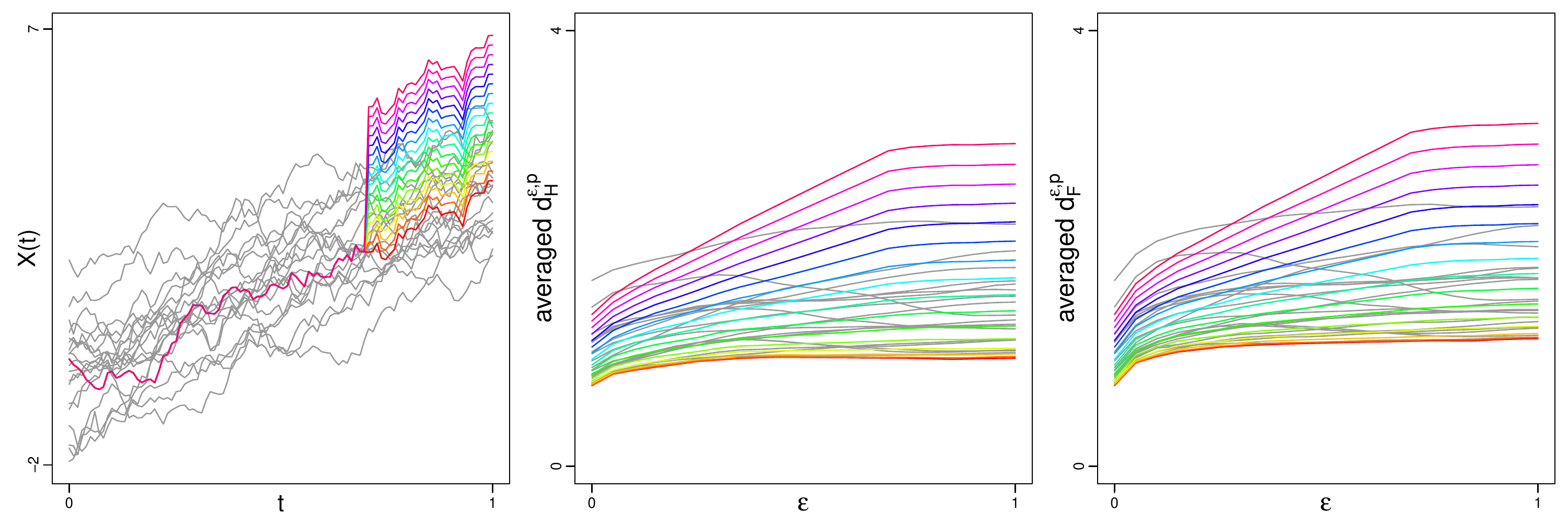}}
    \centerline{\includegraphics[scale=0.4]{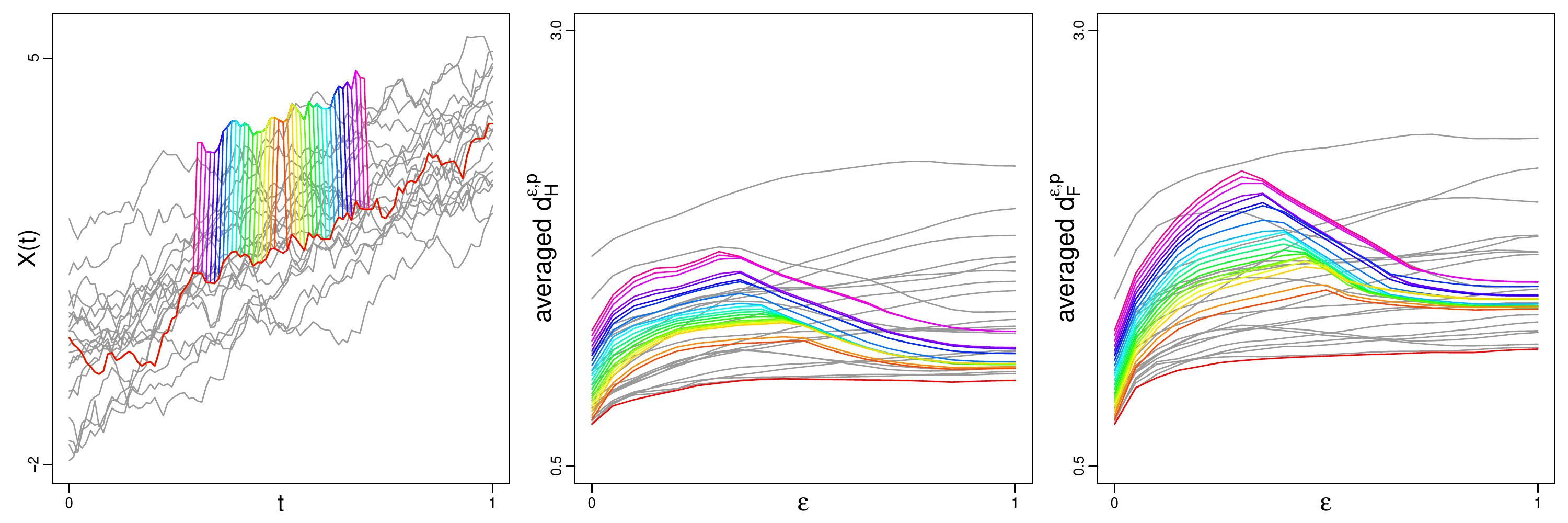}}
    \centerline{\includegraphics[scale=0.4]{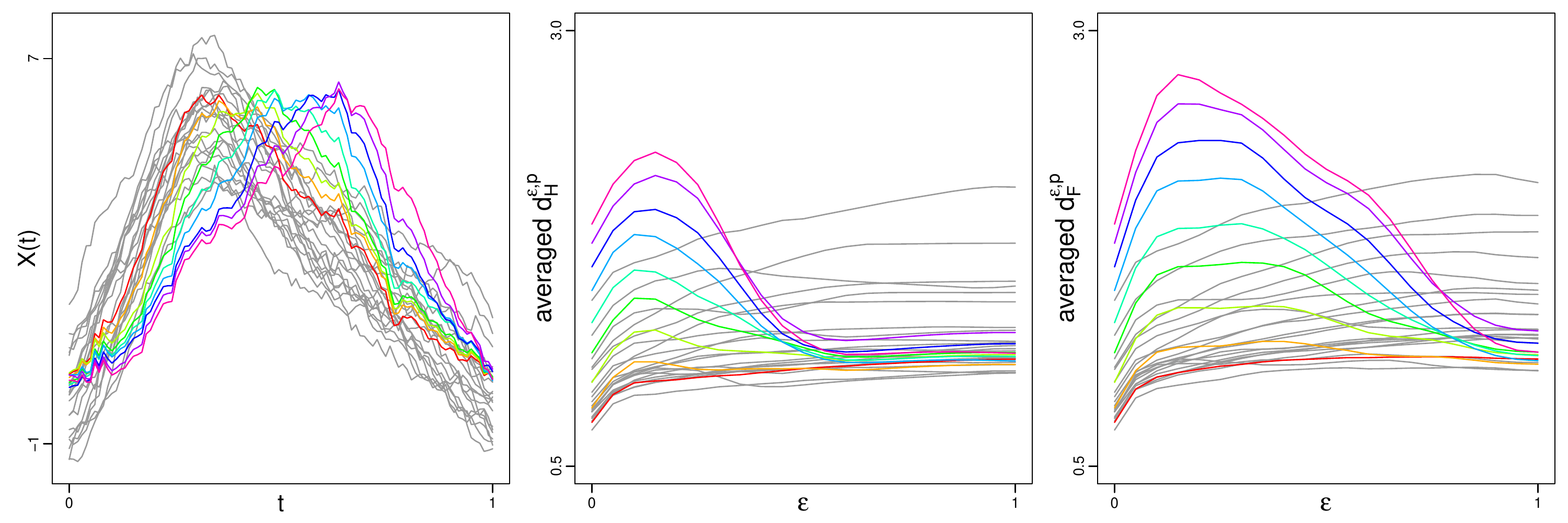}}
    \centerline{\includegraphics[scale=0.4]{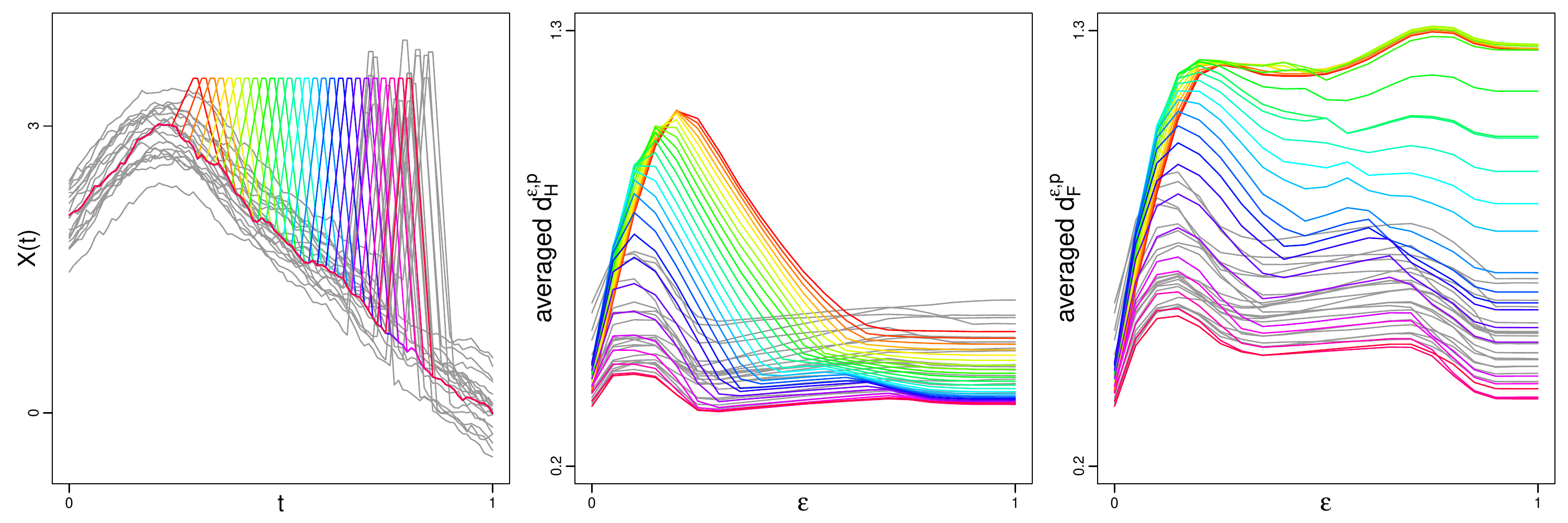}}
    \caption{First column: $20$ i.i.d. realizations from Models $1-4$ (in rows) with contaminating curves displayed in color. Second and third columns: Averaged $d_H^{\epsilon,p}$ and $d_F^{\epsilon,p}$ distances (as functions of $\epsilon\in[0,1]$) between outliers and observations (in color) and between observations (in gray).}
    \label{fig:Models}
\end{figure}

%\begin{figure}[!ht]
%    \centerline{\includegraphics[scale=0.45]{Figures/Model12.pdf}}
%    \caption{Mean $d_H^{\epsilon,p}$ and $d_F^{\epsilon,p}$ distances between outliers and observations in Model 1, with the jump of outliers occurring at $T=0.3$.}
%    \label{fig:Model12}
%\end{figure}

The integrated ball Hausdorff and Fr\'echet distances exhibit a range of interesting behaviors across the different models. Not only do the curves $\epsilon\mapsto d_H^{\epsilon,p}$ and $\epsilon\mapsto d_F^{\epsilon,p}$ associated with outliers allow to separate them from the non-contaminated observations, they also present strikingly different structures to those of the base model. 

In Model 1 (jump outlier), the values of $d_H^{\epsilon,p}$ and $d_F^{\epsilon,p}$ grow steadily from the $L_p$ distance (at $\epsilon = 0$) to the value of the corresponding global metric (at $\epsilon = 1$). 
For both distances, this behavior is due to the outlying domain being included in a larger collection of balls in Definition~\ref{def:integral-ball-metric}, as the radius of the ball increases towards covering the entire domain. 
Furthermore, the outliers clearly display a different behavior from the base observations: the distances $d_H^{\epsilon,p}$ and $d_F^{\epsilon,p}$ initially increase and then reach a plateau, staying nearly constant at the global Hausdorff, resp. Fr\'echet, value.

%\begin{figure}[!ht]
%    \centerline{\includegraphics[scale=0.45]{Figures/Model2.pdf}}
%    \caption{Mean $d_H^{\epsilon,p}$ and $d_F^{\epsilon,p}$ distances between the outliers and the observation of the Model 2, with the height of the peak $H=2.25$.}
%    \label{fig:Model2}
%\end{figure}
%
%\begin{figure}[!ht]
%    \centerline{\includegraphics[scale=0.45]{Figures/Model3.pdf}}
%    \caption{Mean $d_H^{\epsilon,p}$ and $d_F^{\epsilon,p}$ distances between the outliers and the observation of the Model 3.}
%    \label{fig:Model3}
%\end{figure}

Models $2$ and $3$ exhibit interesting peaked behavior. %, that is due to a similar phenomenon as for Model $1$. 
Initially, as $\epsilon$ grows, the outlying peak is included in a larger set of balls, increasing the value of the metric. However, as $\epsilon$ keeps growing, the local matching of the features becomes less strict, allowing the curves to be matched more diagonally. 

In Model $2$, the values of $d_H^{\epsilon,p}$ and $d_F^{\epsilon,p}$ grow from the $L_p$ distance towards a pronounced peak that occurs between $\epsilon = 0.2$ to $\epsilon = 0.5$ (depending on the width of the peak), before settling down at a lower value of the corresponding global metric. When the level of the peak was not clearly above the highest observations of the base model, the global $d_H$ and $d_F$ both did a poor job at picking out the outlying functions. With suitable choices of $\epsilon$, the integrated ball versions were able to single out and separate the outliers much better than their global counterparts. With narrow peak widths, the $L_p$ distance also struggles to notice the outliers.

In Model $3$ the peaked behavior is extremely pronounced. As the scale of the $y$-axis is larger than that of the $x$-axis, it becomes less penalizing to match the curves horizontally rather than vertically. In this case, this behavior is very prominent, as the global versions of $d_H$ and $d_F$ measure the distances between the curves almost exclusively horizontally. This kind of behavior can be beneficial, for example in applications where the variation in phase is of lesser concern than the variation in amplitude. In Model 3, the integrated versions $d_H^{\epsilon,p}$ and $d_F^{\epsilon,p}$ allow for better control in the locality of this ``feature matching'' property. %Furthermore, when the function values are scaled such that the axes have more similar scaling, this peaked behavior becomes less pronounced but does not vanish (see the top panels of Figure \ref{fig:Motivating}).

In Models 2 and 3, the shapes of the integrated ball metric curves allow to discriminate immediately between outliers and noncontaminated curves: the outliers produce redescending distances as $\epsilon$ grows to $1$.

In Model $4$, the global Hausdorff distance fails to isolate the outliers, even when the timing of the sharp feature is extremely misplaced compared to the base model. 
As opposed to Model 3, the $L_p$ distances do not detect the outliers either, while the global Fr\'echet distance does. 
In both cases, however, the integrated ball versions exhibit the interesting peaked behavior, seen in Model $2$ and Model $3$ as well. 

%The global Fr\'echet distance does a much better job at finding the outliers. 
%
%
%As the misplacement of the sharp feature became extreme, approaching the timing of the earlier smoother peak, the $L_p$ end of the $\epsilon$ spectrum started to also rank the outliers as more central than even ones with the peak within the range of that of the base model. In both cases, with a number of suitable values of $\epsilon$, the integrated ball distances maintained a good balance between separating the outliers and the local matching of the sharp mistimed feature.
%\begin{figure}[!ht]
%    \centerline{\includegraphics[scale=0.45]{Figures/Model4.pdf}}
%    \caption{Mean $d_H^{\epsilon,p}$ and $d_F^{\epsilon,p}$ distances between the outliers and the observation of the Model 4.}
%    \label{fig:Model4}
%\end{figure}

%%%%%%%%%%%%%%%%%%%%%%%%%%%
%%%%%%%%%%%%%%%%%%%%%%%%%%%
%%%%%%%%%%%%%%%%%%%%%%%%%%%

\section{Discussion}

This paper provides an integrated version of any pseudometric $d$. Naturally, the integrated version $d^{\epsilon,p}$ of $d$ inherits some properties of its underlying pseudometric $d$ as we illustrate in Theorems~\ref{thm:integral-ball-key-properties} and~\ref{thm:discretisation-convergence}. An exhaustive study of the links between integrated and base metric is beyond the scope of this paper. We, however, provide a few avenues of research that might prove fruitful in the future.

Using pseudometrics in FDA is of interest in many statistical problems. Section~\ref{sec:simulated} illustrates this in an outlier detection setting. Of course, such metrics are of interest in other problems, such as nonparametric regression (where metrics can be used in functional kernel regression), clustering, and classification. See, for example, the review papers~\citet{AneirosEtAl2019b,AneirosEtAl2019a,Cuevas2014} and \citet{GoiaAndVieu2016}.

In the regression setting, asymptotic analysis requires studying small ball probabilities (SBPs). Very little is known for Fr\'echet and Hausdorff distances in that respect. Such analyses are therefore promising, yet technically challenging prospects for future research. In a more general framework, one might be interested in deriving functional SBPs based on corresponding results for the underlying (pseudo)metric $d$. For existing results on SBPs in FDA or the setting of Gaussian processes, see~\citet{BongiornoEtAl2017,FerratyAndVieu2006} and \citet{LiAndShao2001}.

Finally, the results derived in this paper deal with a general pseudometric $d$. In practical applications, one might wish to select $d$ adaptively to the problem at hand. While there is no universal optimal selection procedure for $d$, crossvalidation (choosing $d$ in a given set of pseudometrics) could be applied in selected applications, such as classification or regression. Naturally, regardless of the choice of $d$, crossvalidation can be conducted on $\epsilon$ and/or $p$ as well.

%%%%%%%%%%%%%%%%%%%%%%%%%%%
%%%%%%%%%%%%%%%%%%%%%%%%%%%
%%%%%%%%%%%%%%%%%%%%%%%%%%%

\section*{Acknowledgements}

\noindent S. Helander wishes to thank the Emil Aaltonen Foundation (grant 200033 N1). P.~Laketa was supported by the OP RDE project ``International mobility of research, technical and administrative staff at the Charles University" CZ.02.2.69/0.0/0.0/18\_053/0016976. The work of S. Nagy and P. Laketa was supported by the grant 19-16097Y of the Czech Science Foundation, and by the PRIMUS/17/SCI/3 project of Charles University. G. Van Bever thanks the Fond National pour la Recherche Scientifique (FNRS) for their support through the CDR grant J.0208.20. 

%%%%%%%%%%%%%%%%%%%%%%%%%%%
%%%%%%%%%%%%%%%%%%%%%%%%%%%
%%%%%%%%%%%%%%%%%%%%%%%%%%%

%%%%%%%%%%%%%%%%%%%%%%%%%%%
%%%%%%%%%%%%%%%%%%%%%%%%%%%
%%%%%%%%%%%%%%%%%%%%%%%%%%%
\end{document}